\documentclass[12pt, reqno]{amsart}

\usepackage{amsmath, amsthm, amscd, amsfonts, amssymb, graphicx, color}
\usepackage[bookmarksnumbered, colorlinks, plainpages]{hyperref}
\input{mathrsfs.sty}
\hypersetup{colorlinks=true,linkcolor=red, anchorcolor=green,
citecolor=cyan, urlcolor=red, filecolor=magenta, pdftoolbar=true}

\newtheorem{theorem}{Theorem}[section]
\newtheorem{lemma}[theorem]{Lemma}

\newtheorem{corollary}[theorem]{Corollary}
\theoremstyle{definition}
\newtheorem{definition}[theorem]{Definition}
\newtheorem{example}[theorem]{Example}

\theoremstyle{remark}
\newtheorem{remark}[theorem]{Remark}
\numberwithin{equation}{section}

\begin{document}

\title[Approximately angle preserving mappings]{Approximately angle preserving mappings}

\author{Mohammad Sal Moslehian$^*$}
\address{Department of Pure Mathematics, Ferdowsi University of Mashhad, P.O. Box 1159, Mashhad 91775, Iran}
\email{moslehian@um.ac.ir}

\author{Ali Zamani}
\address{Department of Mathematics, Farhangian University, Tehran, Iran}
\email{zamani.ali85@yahoo.com}

\author{Pawe\l{}  W\'ojcik}
\address{Institute of Mathematics, Pedagogical University of Cracow, Podchor\c a\.zych 2, 30-084 Krak\'ow, Poland.}
\email{pawel.wojcik@up.krakow.pl}


\begin{abstract}
In this paper, we present some characterizations of linear mappings, which preserve vectors at a specific angle. We introduce the concept of $(\varepsilon, c)$-angle preserving mappings for $|c|<1$ and $0\leq \varepsilon < 1 + |c|$. In addition, we define $\widehat{\varepsilon}\,(T, c)$ as the ``smallest'' number $\varepsilon$ for which $T$ is $(\varepsilon, c)$-angle preserving mapping.
We state some properties of the function $\widehat{\varepsilon}\,(., c)$, and then propose an exact formula for $\widehat{\varepsilon}\,(T, c)$ in terms of the norm $\|T\|$ and the minimum modulus $[T]$ of $T$. Finally, we characterize the approximately angle preserving mappings. 
\end{abstract} \maketitle

\section{Introduction}
Throughout this paper, let $\mathscr{H}, \mathscr{K}$ denote real Hilbert spaces with dimensions
greater than or equal to $2$ and let $\mathbb{B}(\mathscr{H}, \mathscr{K})$ denote the Banach space
of all bounded linear mappings between Hilbert spaces $\mathscr{H}$ and $\mathscr{K}$.
We write $\mathbb{B}(\mathscr{H})$ for $\mathbb{B}(\mathscr{H}, \mathscr{H})$.

As usual, vectors $x, y\in \mathscr{H}$ are said to be orthogonal, $x \perp y$,
if $\langle x, y\rangle = 0$, where $\langle ., .\rangle$ denotes the inner product of $\mathscr{H}$.
A mapping $T: \,\mathscr{H}\longrightarrow \mathscr{K}$ is
called orthogonality preserving if it preserves orthogonality, that is
$$x \perp y \,\Longrightarrow \, Tx \perp Ty \qquad(x, y\in \mathscr{H}).$$
It is known that orthogonality preserving mappings may be nonlinear and discontinuous; cf. \cite{Ch.1}.
Under the additional assumption of linearity, a mapping $T$ is orthogonality preserving
if and only if it is a scalar multiple of an isometry, that is $T = \gamma U$,
where $U$ is an isometry and $\gamma\geq 0$; see \cite{Ko}.

It is natural to consider approximate orthogonality ($\varepsilon$-orthogonality)
$x \perp ^{\varepsilon} y$ defined by $|\langle x, y\rangle|\leq \varepsilon \|x\|\,\|y\|.$
For $\varepsilon\geq 1$, it is clear that every pair of vectors are $\varepsilon$-orthogonal,
so the interesting case is when $\varepsilon\in[0, 1)$.

One can consider the class of approximately orthogonality preserving mappings.
A mapping $T: \,\mathscr{H}\longrightarrow \mathscr{K}$
is said to be approximately orthogonality preserving mapping,
or $\varepsilon$-orthogonality preserving mapping, if
$$x \perp y \,\Longrightarrow \,Tx \perp^{\varepsilon} Ty \qquad (x, y \in \mathscr{H}).$$
Obviously, if $\varepsilon = 0$, then $T$ is orthogonality preserving.
Hence, the natural question is whether an $\varepsilon$-orthogonality preserving linear mapping
$T$ must be close to a linear orthogonality preserving mapping, cf. \cite{C.L, C.L.W, Tu}.

In a Hilbert space $\mathscr{H}$ we define a relation connected to the notion of angle.
Fix $c\in (-1, 1)$. For $x, y\in \mathscr{H}$ we say $x\,\angle_c\,y$ if
$\langle x, y\rangle = c\,\|x\|\,\|y\|$. So that $c= \cos (\alpha)$ where $\alpha$
is the angle between $x$ and $y$ if $x, y \in \mathscr{H}\setminus\{0\}$.

A mapping $T: \,\mathscr{H}\longrightarrow \mathscr{K}$ is
called $c$-angle preserving if it preserves angle, that is
$$x\,\angle_c\,y \,\Longrightarrow \, Tx\,\angle_c\, Ty \qquad(x, y\in \mathscr{H}).$$

Angle preserving mappings may be very far from linear and continuous mappings.
Actually, there exist an (infinite-dimensional) Euclidean space $\mathscr{H}$
and an injective map $T: \,\mathscr{H}\longrightarrow \mathscr{H}$ such that the condition
$x\,\angle_\frac{1}{2}\,y$ implies that $Tx\,\angle_\frac{1}{2}\, Ty$,
while the map $T$ is discontinuous at all points (see \cite[Remark 3]{Ku}).

A characterization of angle preserving mappings on finite-dimensional Euclidean spaces
was obtained in \cite{Ku}. Further, Chmieli\'{n}ski \cite{Ch.2} studied stability of angle
preserving mappings on the plane. Recently, angle preserving mappings have been studied
in \cite{Wo.2, Wo.3, M.Z.F}.

In the next section, we will present some characterizations of linear mappings preserving some angles.
We will show (Theorem \ref{th.00}) that a nonzero linear map $T$ is $c$-angle preserving
if and only if $T$ is a scalar multiple of an isometry. In fact, this result is a generalization
of \cite[Theorem 1]{Ch.1} and \cite[Theorem 3.8]{M.Z.F}.

Let us fix $\varepsilon\in[0, 1)$ and define $x\,\angle_c^\varepsilon\,y$ as
$$\Big|\langle x, y\rangle - c\,\|x\|\,\|y\|\Big|\leq \varepsilon \|x\|\,\|y\|,$$
which is equivalent to $c-\varepsilon \leqslant \cos \alpha \leqslant c+\varepsilon$,
where $\alpha$ is the angle between $x$ and $y$.

Obviously, if $ c = 0$, then $\angle_0 = \perp $ and $\angle_0^\varepsilon = \perp^{\varepsilon} $.
Moreover, it is easy to see that $\angle_c$ and $\angle_c^\varepsilon$ are weakly homogeneous in
the sense that $x\,\angle_c \,y\, \Leftrightarrow \, \alpha x\,\angle_c \, \beta y$ and
$x\,\angle_c^\varepsilon\, y\, \Leftrightarrow \, \alpha x\,\angle_c^\varepsilon\, \beta y$
for all $\alpha, \beta\in\mathbb{R^{+}}$.

Also, for $\varepsilon\geq 1 + |c|$, it is obvious that $x\,\angle_c^\varepsilon\,y$ for all
$x, y\in \mathscr{H}$. Hence, we shall only consider the case $\varepsilon\in[0, 1 + |c|)$.

A mapping $T: \,\mathscr{H}\longrightarrow \mathscr{K}$ satisfying the condition
$$x \,\angle_c \,y \,\Longrightarrow \,Tx \,\angle_c^\varepsilon\, Ty \qquad (x, y \in \mathscr{H})$$
is called an $\varepsilon$-approximately $c$-angle preserving mapping, or $(\varepsilon, c)$-angle preserving mapping.

Recently, angle preserving mappings have been studied in \cite{Wo.2, Wo.3} via an approach different from ours.
When $\mathscr{H}, \mathscr{K}$ are finite-dimensional, the third author \cite{Wo.2} proved that for an arbitrary
$\delta>0$ there exists $\varepsilon>0$ such that for any linear $(\varepsilon, c)$-angle preserving mapping
$T$ there exists a linear $c$-angle preserving mapping such that
$$\|T - S\| \leq \delta \min\{\|T\|, \|S\|\}.$$

Our next intention is to obtain a characterization of the approximately preserving angle mappings.
Notice that if $0\leq \varepsilon_1 \leq \varepsilon_2 < 1 + |c|$ and $T$ is
$(\varepsilon_1, c)$-angle preserving mapping, then $T$ is $(\varepsilon_2, c)$-angle preserving mapping as well.
This fact motivates us to give the following definition; see also \cite{Z.C.H.K}.
\begin{definition}\label{df.2}
Let $c\in (-1,1)$. For each map $T: \,\mathscr{H}\longrightarrow \mathscr{K}$, let $\widehat{\varepsilon}\,(T, c)$
be the ``smallest'' number $\varepsilon$ such that $T$ is $(\varepsilon, c)$-angle preserving. Hence
\begin{eqnarray*}
\widehat{\varepsilon}\,(T, c) := \inf\Big\{\varepsilon \in[0, 1 + |c|]: \,T
\mbox{~is~ $(\varepsilon, c)$-angle~ preserving~mapping}\Big\}.
\end{eqnarray*}
\end{definition}
Thus $\widehat{\varepsilon}\,(T, c) = 1 + |c|$ whenever $T$ is not an approximately $c$-angle preserving mapping.
Also, it is easy to see that $\widehat{\varepsilon}\,(T, - c) = \widehat{\varepsilon}\,(T, c)
= \widehat{\varepsilon}\,(\alpha T, c)$ for all $\alpha \in \mathbb{R}\setminus\{0\}$.

In the last section we intend to state some basic properties of the function $\widehat{\varepsilon}\,(., c)$.
If $T\in \mathbb{B}(\mathscr{H}, \mathscr{K})$, then we propose an exact formula for
$\widehat{\varepsilon}\,(T, c)$ in terms of the norm $\|T\|$ and the minimum modulus $[T]$ of $T$.
Here $[T]$ is defined to be the largest number $m\geq 0$ such that $\|Tx\| \geq m \|x\|\,\,(x \in \mathscr{H})$.
We then use this formula to characterize the approximately $c$-angle preserving mappings (Corollary \ref{cr.4}).
Actually, we show that every nonzero linear mapping $T$ is the approximately $c$-angle
preserving mapping if and only if $T$ is bounded below.

\section{Linear mappings preserving angles}

We start our work with the following lemmas.
The first lemma follows immediately from the definition of the angle between vectors.
\begin{lemma}\label{le.2}
Let $c\in [0, 1)$. If $x, y \in\mathscr{H}$ such that $\|x\| = \|y\| = 1$ and $x\perp y$, then
\begin{enumerate}
\item[(i)] $\big(x + \sqrt{\frac{1 + c}{1 - c}} y\big)\, \angle_c \,\big(- x + \sqrt{\frac{1 + c}{1 - c}} y\big)$.\\

\item[(ii)] $\big(x + \sqrt{\frac{1 - c}{1 + c}} y\big)\, \angle_c \,\big(x - \sqrt{\frac{1 - c}{1 + c}} y\big)$.
\end{enumerate}
\end{lemma}
Let us quote a result from \cite{Wo.1}.
\begin{lemma}\cite[Theorem 2.3]{Wo.1}\label{le.3}
Let $T\in \mathbb{B}(\mathscr{H}, \mathscr{K})$ be an injective linear map.
Suppose that $\dim \mathscr{H} = n.$ Then there exist an orthonormal basis
$\{x_1, x_2, ..., x_n\}$ for $\mathscr{H}$ such that
$$[T] = \|Tx_1\|, \quad \|Tx_2\| = \|T\| \quad \mbox{and} \quad  Tx_i \perp Tx_j \quad (1 \leq i\neq j \leq n).$$
\end{lemma}
As a consequence of Lemma \ref{le.3}, the following result immediately follows.
\begin{corollary}\label{cr.3.5}
Let $T: \,\mathscr{H}\longrightarrow \mathscr{K}$ be a nonzero injective linear map.
Suppose that unit vectors $x, y\in \mathscr{H}$ are linearly independent.
Then there exist unit vectors $x_1, x_2$ such that
$$x_1\perp x_2, \quad Tx_1 \perp Tx_2, \quad \|Tx_1\| \leq \|Tx\| \leq \|Tx_2\| \quad \mbox{and} \quad \|Tx_1\| \leq \|Ty\| \leq \|Tx_2\|.$$
\end{corollary}

We are now ready to characterize the $c$-angle preserving mappings.
In fact, the following result is a generalization of \cite[Theorem 1]{Ch.1}.

\begin{theorem}\label{th.00}
Let $T: \,\mathscr{H}\longrightarrow \mathscr{K}$ be a nonzero linear map and let $c\in (-1,1)$.
Then the following statements are mutually equivalent:
\begin{enumerate}
\item[(i)] $x\, \angle_c \,y\, \Longrightarrow \,Tx\, \angle_c \,Ty \qquad (x, y\in \mathscr{H}).$

\item[(ii)] There exists $\gamma>0$ such that $\|Tx\|=\gamma\|x\| \qquad (x\in \mathscr{H}).$
\end{enumerate}
\end{theorem}
\begin{proof}
The implication (ii)$\Rightarrow$(i) follows from the polarization formula. We do not prove the implication
(i)$\Rightarrow$(ii).  We will prove a more general theorem (see Corollary \ref{cr.4.5}).
\end{proof}
The following example shows that Theorem \ref{th.00} fails if the assumption of linearity is dropped.
Moreover, nonlinear mappings satisfying
$x\,\angle_c\,y \,\Longrightarrow \, Tx\,\angle_c\, Ty\ \ (x, y\in \mathscr{H})$ may be
very strange, even noncontinuous.
\begin{example}
Let $c\in (-1,1)$. Let $\varphi\colon\mathscr{H}\to \mathbb{R}$ be fixed nonvanishing function.
Then for the mapping $T:\mathscr{H}\longrightarrow \mathscr{H}$
defined by $T(x):=\varphi(x)\!\cdot\! x$ we have $x\, \angle_c \,y\, \Longrightarrow \,Tx\, \angle_c \,Ty$
for all $x, y\in \mathscr{H}$. If $\varphi$ is not continuous, then $T$ clearly is not continuous. In particular,
$T$ clearly is not a similarity.
\end{example}
Taking $\mathscr{K} = \mathscr{H}$ and
$T = id: (\mathscr{H}, {\langle ., .\rangle}_1)\longrightarrow (\mathscr{H}, {\langle ., .\rangle}_2)$
the identity map, one obtains, from Theorem \ref{th.00}, the following result.

\begin{corollary}\label{cr.9}
Let $c\in (-1,1)$. Suppose that $\mathscr{H}$ is a vector space equipped with two (complete) inner products
${\langle ., .\rangle}_1$ and ${\langle ., .\rangle}_2$ generating the norms
${\|.\|}_1$, ${\|.\|}_2$ and $c$-angle relations $\, \angle_{c,1} \,$, $\, \angle_{c,2} \,$,
respectively. Then the following conditions are equivalent:
\begin{itemize}
\item[(i)] There exists $\gamma>0$ such that ${\|x\|}_2=\gamma{\|x\|}_1 \qquad (x\in \mathscr{H})$.

\item[(ii)] $x\, \angle_{c,1} \,y \,\, \Longrightarrow \,\, x\, \angle_{c,2} \,y \qquad (x, y\in \mathscr{H})$.

\item[(iii)] $\sup\left\{\big|\frac{{\langle x, y\rangle}_2}{{\|x\|}_2\,{\|y\|}_2} - c\big|;
\,\,x \,\angle_{c,1}\, y,\, x, y\in\mathscr{H}\setminus\{0\}\right\} = 0$.
\end{itemize}
\end{corollary}

\begin{corollary}
Let $T\in \mathbb{B}(\mathscr{H}, \mathscr{K})$ be a bijective linear map and let $c\in (-1,1)$.
Then the following statements are mutually equivalent:
\begin{enumerate}
\item[(i)] $x\, \angle_c \,y\, \Longrightarrow \,Tx\, \angle_c \,Ty \qquad (x, y\in \mathscr{H}).$

\item[(ii)] $\|TST^{-1}\| \leq \|S\|$ for all invertible linear mappings $S\in \mathbb{B}(\mathscr{H}).$
\end{enumerate}
\end{corollary}
\begin{proof}
(i)$\Rightarrow$(ii) This implication follows immediately from the equivalence
(i)$\Leftrightarrow$(ii) of Theorem \ref{th.00}.

(ii)$\Rightarrow$(i) Suppose that (ii) holds. For every $\varepsilon> 0$, we have
$$\Big\|\varepsilon I + T(x\otimes y)T^{-1}\Big\| = \Big\|T\big(\varepsilon I + x\otimes y\big)T^{-1}\Big\|
\leq \Big\|\varepsilon I + x\otimes y\Big\|\qquad (x, y\in \mathscr{H}).$$
Here, $x\otimes y$ denotes the rank one operator in $\mathbb{B}(\mathscr{H})$
defined by $(x\otimes y)(z) : = \langle z, y\rangle x$ for $z\in\mathscr{H}$.
Letting $\varepsilon\rightarrow 0^{+}$ we obtain
$$\big\|T(x\otimes y)T^{-1}\big\| \leq \big\|x\otimes y\big\| \qquad (x, y\in \mathscr{H}).$$
This implies $\|T\|\|T^{-1}\|\leq 1$.
Hence
$$\|T\|\|x\| \leq \frac{\|x\|}{\|T^{-1}\|} \leq \|Tx\| \leq \|T\|\|x\| \qquad (x\in \mathscr{H}),$$
which yields that
$$\|Tx\| = \|T\|\|x\| \qquad (x\in \mathscr{H}).$$
Now, by the equivalence (i)$\Leftrightarrow$(ii) of Theorem \ref{th.00}, we get (i).
\end{proof}
\section{Approximately preserving angle mappings}

Our aim in this section is to characterize the approximately preserving angle mappings.
We start our work with the following lemma. It follows immediately from Definition \ref{df.2}.
\begin{lemma}\label{le.1}
Let $T: \,\mathscr{H}\longrightarrow \mathscr{K}$ be a linear map and let $c\in (-1,1)$.
Then the following statements hold:
\begin{enumerate}
\item[(i)] $\widehat{\varepsilon}\,(T, c) = \sup\left\{\big|\frac{\langle Tx, Ty\rangle}
{\|Tx\|\,|Ty\|} - c\big|; \,\,x\, \angle_c\, y,\, x, y\in\mathscr{H}\setminus\{0\}\right\}$.\\

\item[(ii)] $\widehat{\varepsilon}\,(T, c) = \sup\left\{\big|\frac{\langle Tx, Ty\rangle}
{\|Tx\|\,|Ty\|} - c\big|; \,\,x\, \angle_c\, y,\, \|x\| = \|y\| =1, \,x, y\in\mathscr{H}\right\}$.
\end{enumerate}
\end{lemma}
Our next theorem is a generalization of \cite[Lemma 2.2]{Z.C.H.K}.
\begin{theorem}\label{th.3}
Let $T: \,\mathscr{H}\longrightarrow \mathscr{K}$ be a nonzero linear map and let $c\in (-1,1)$.
If $[T] = 0$, then $\widehat{\varepsilon}\,(T, c) = 1 + |c|.$
\end{theorem}
\begin{proof}
Since $\widehat{\varepsilon}\,(T, - c) = \widehat{\varepsilon}\,(T, c)$,
we may assume that $c\in[0, 1)$. We consider two cases.\\

Case 1. $T$ is not injective.

Then there exists a subspace $\mathscr{H}_1$ such that $2\leq\dim \mathscr{H}_1< \infty$
and $T|_{\mathscr{H}_1}$ is not injective, i.e., $\{0\} \neq \ker(T|_{\mathscr{H}_1}) \neq \mathscr{H}_1$.
(Indeed, if $T$ is injective on every finite-dimensional subspace, then $T$ has to be injective.)
Since the set $\ker(T|_{\mathscr{H}_1})$ is not dense,
so we can find two vectors
$x\in \big(\ker(T|_{\mathscr{H}_1})\big)^\perp$, $y\in \ker(T|_{\mathscr{H}_1})$ such that
$\|x\| = \|y\| = 1$ and $x\perp y$. Now, by Lemma \ref{le.2} (i),
$(x + \sqrt{\frac{1 + c}{1 - c}} y)\, \angle_c \,(- x + \sqrt{\frac{1 + c}{1 - c}} y)$. We have
\begin{align*}
\left|\frac{\langle T(x + \sqrt{\frac{1 + c}{1 - c}} y) , T(- x + \sqrt{\frac{1 + c}{1 - c}} y)\rangle}
{\left\|T(x + \sqrt{\frac{1 + c}{1 - c}} y)\right\|\,\left\|T(- x + \sqrt{\frac{1 + c}{1 - c}} y)\right\|} - c \right|
= \left|\frac{-\|Tx\|^2}{\|Tx\|^2} - c \right| = 1 + c.
\end{align*}
Thus, by Lemma \ref{le.1} (i), $\widehat{\varepsilon}\,(T, c) = 1 + c $. \\

Case 2. $T$ is injective.

Assume that $\widehat{\varepsilon}\,(T, c) < 1 + c $. Then there exists $\varepsilon_0 < 1 + c$
such that $T$ is an $(\varepsilon_0, c)$-angle preserving mapping. Now, consider arbitrarily unit vectors
$x, y \in \mathscr{H}$. If $x$ and $y$ are linearly dependent, then
$\sqrt{\frac{(1 - c)(1 + c - \varepsilon_0)}{(1 + c)(1 - c + \varepsilon_0)}}\|Ty\| \leq \|Tx\|$ obviously holds.
Also, if $x$ and $y$ are linearly independent, then by Corollary \ref{cr.3.5},
there exist unit vectors $x_1, x_2$ such that
\begin{align}\label{id.3.1}
x_1\perp x_2, \quad Tx_1 \perp Tx_2,  \quad \|Tx_1\| \leq \|Tx\|\leq \|Tx_2\|
\quad \mbox{and} \quad \|Tx_1\| \leq \|Ty\| \leq \|Tx_2\|.
\end{align}
So, by Lemma \ref{le.2} (ii), $\Big(x_1 + \sqrt{\frac{1 - c}{1 + c}} x_2\Big)\, \angle_c \,\Big(x_1 - \sqrt{\frac{1 - c}{1 + c}} x_2\Big)$,
whence
$$T\left(x_1 + \sqrt{\frac{1 - c}{1 + c}} x_2\right)\, \angle_c^{\varepsilon_0} \,T\left(x_1 - \sqrt{\frac{1 - c}{1 + c}} x_2\right).$$
Let us put $u = x_1 + \sqrt{\frac{1 - c}{1 + c}} x_2$ and $v = x_1 - \sqrt{\frac{1 - c}{1 + c}} x_2$.
Therefore
\begin{align*}
-\|Tx_1\|^2  + \frac{1 - c}{1 + c} \|Tx_2\|^2 &+ c\left(\|Tx_1\|^2  + \frac{1 - c}{1 + c} \|Tx_2\|^2\right)
\\& \leq \left|\|Tx_1\|^2  - \frac{1 - c}{1 + c} \|Tx_2\|^2
- c\left(\|Tx_1\|^2  + \frac{1 - c}{1 + c} \|Tx_2\|^2\right)\right|
\\ & = \Big|\langle Tu , Tv \rangle - c\|Tu\|\,\|Tv\|\Big|
\\ & \leq \varepsilon_0\|Tu\|\,\|Tv\|
\\& = \varepsilon_0\left\|T\Big(x_1 + \sqrt{\frac{1 - c}{1 + c}} x_2\Big)\right\|\,
\left\|T\Big(x_1 - \sqrt{\frac{1 - c}{1 + c}} x_2\Big)\right\|
\\ & \leq \varepsilon_0\left(\|Tx_1\|^2  + \frac{1 - c}{1 + c} \|Tx_2\|^2\right).
\end{align*}
Hence
$$-\|Tx_1\|^2  + \frac{1 - c}{1 + c} \|Tx_2\|^2 + c\left(\|Tx_1\|^2  + \frac{1 - c}{1 + c} \|Tx_2\|^2\right)
\leq \varepsilon_0\left(\|Tx_1\|^2  + \frac{1 - c}{1 + c} \|Tx_2\|^2\right),$$
or equivalently,
\begin{align}\label{id.3.2}
\sqrt{\frac{(1 - c)(1 + c - \varepsilon_0)}{(1 + c)(1 - c + \varepsilon_0)}}\|Tx_2\| \leq \|Tx_1\|.
\end{align}
Employing (\ref{id.3.1}) and (\ref{id.3.2}) we reach
$$\sqrt{\frac{(1 - c)(1 + c - \varepsilon_0)}{(1 + c)(1 - c + \varepsilon_0)}}\|Ty\|
\leq \sqrt{\frac{(1 - c)(1 + c - \varepsilon_0)}{(1 + c)(1 - c + \varepsilon_0)}}\|Tx_2\| \leq \|Tx_1\| \leq \|Tx\|.$$
By passing to the supremum over $y$ and passing to the infimum over $x$ in the above inequality we obtain
$\sqrt{\frac{(1 - c)(1 + c - \varepsilon_0)}{(1 + c)(1 - c + \varepsilon_0)}}\|T\| \leq [T]$.
Since $\|T\| > 0$ and $[T] = 0$ we get $\varepsilon_0 = 1 + c$.
This contradiction shows that $\widehat{\varepsilon}\,(T, c) = 1 + c $.
\end{proof}
Next, we formulate one of our main results.
\begin{theorem}\label{th.4}
Let $c\in (-1,1)$. Suppose that $T\in \mathbb{B}(\mathscr{H}, \mathscr{K})$ and $[T] \neq 0$. Then
$$\widehat{\varepsilon}\,(T, c) =
\frac{(1 - |c|^2)(\|T\|^2 - [T]^2)}{(1 + |c|)\|T\|^2 + (1 - |c|)[T]^2}.$$
\end{theorem}
\begin{proof}
We may assume that $c\in[0, 1)$. Since $[T]>0$, there exist unit vectors $x_1, x_2$ such that
\begin{align}\label{id.3.3}
x_1\perp x_2, \quad Tx_1 \perp Tx_2, \quad [T] = \|Tx_1\| \quad \mbox{and} \quad \|Tx_2\| = \|T\|.
\end{align}
It follows from Lemma \ref{le.2} (ii), that $\Big(x_2 + \sqrt{\frac{1 - c}{1 + c}} x_1\Big)\,
\angle_c \,\Big(x_2 - \sqrt{\frac{1 - c}{1 + c}} x_1\Big)$. Let us put
$u = x_2 + \sqrt{\frac{1 - c}{1 + c}} x_1$ and $v = x_2 - \sqrt{\frac{1 - c}{1 + c}} x_1$.
Therefore, by (\ref{id.3.3}), we have
\begin{align*}
\left|\frac{\langle Tu , Tv\rangle}{\|Tu\|\,\|Tv\|} - c\right|
& = \left|\frac{\|Tx_2\|^2 - \frac{1 - c}{1 + c}\|Tx_1\|^2}{\|Tx_2\|^2 + \frac{1 - c}{1 + c}\|Tx_1\|^2} - c\right|
\\& = \left|\frac{\|T\|^2 - \frac{1 - c}{1 + c}[T]^2}{\|T\|^2 + \frac{1 - c}{1 + c}[T]^2} - c\right|
\\& = \left|\frac{(1 + c)\|T\|^2 - (1 - c)[T]^2}{(1 + c)\|T\|^2 + (1 - c)[T]^2} - c\right|
\\& = \frac{(1 - c^2)(\|T\|^2 - [T]^2)}{(1 + c)\|T\|^2 + (1 - c)[T]^2}.
\end{align*}
Now, Lemma \ref{le.1} (i) yields that
\begin{align}\label{id.3.110}
\widehat{\varepsilon}\,(T, c) \geq \frac{(1 - c^2)(\|T\|^2 - [T]^2)}{(1 + c)\|T\|^2 + (1 - c)[T]^2}.
\end{align}
On the other hand, let $x, y\in \mathscr{H}$ such that $x\, \angle_c \, y$ and $\|x\| = \|y\| = 1$.
We have
\begin{align*}
\left\|\frac{Tx}{\|Tx\|} + \frac{Ty}{\|Ty\|}\right\|^2 &
= \left\|T\left(\frac{x}{\|Tx\|} + \frac{y}{\|Ty\|}\right)\right\|^2
\\ & \leq \|T\|^2\left\|\frac{x}{\|Tx\|} + \frac{y}{\|Ty\|}\right\|^2
\\ & = \|T\|^2\left(\frac{1}{\|Tx\|^2} + \frac{1}{\|Ty\|^2} + \frac{2c}{\|Tx\|\,\|Ty\|}\right),
\end{align*}
whence
\begin{align}\label{id.3.11}
\left\|\frac{Tx}{\|Tx\|} + \frac{Ty}{\|Ty\|}\right\|^2
\leq \|T\|^2\left(\frac{1}{\|Tx\|^2} + \frac{1}{\|Ty\|^2} + \frac{2c}{\|Tx\|\,\|Ty\|}\right).
\end{align}
Similarly,
\begin{align*}
\left\|\frac{Tx}{\|Tx\|} - \frac{Ty}{\|Ty\|}\right\|^2
\geq [T]^2\left(\frac{1}{\|Tx\|^2} + \frac{1}{\|Ty\|^2} - \frac{2c}{\|Tx\|\,\|Ty\|}\right).
\end{align*}
and
\begin{align*}
\left\|\frac{Tx}{\|Tx\|} - \frac{Ty}{\|Ty\|}\right\|^2
\leq \|T\|^2\left(\frac{1}{\|Tx\|^2} + \frac{1}{\|Ty\|^2} - \frac{2c}{\|Tx\|\,\|Ty\|}\right).
\end{align*}
Now, let
\begin{align}\label{id.3.12}
\left\|\frac{Tx}{\|Tx\|} - \frac{Ty}{\|Ty\|}\right\|^2
= \mu[T]^2\left(\frac{1}{\|Tx\|^2} + \frac{1}{\|Ty\|^2} - \frac{2c}{\|Tx\|\,\|Ty\|}\right)
\end{align}
with $1\leq \mu \leq \frac{\|T\|}{[T]}$. It follows from (\ref{id.3.11}) and (\ref{id.3.12}) that
\begin{align*}
4 & = \left\|\frac{Tx}{\|Tx\|} + \frac{Ty}{\|Ty\|}\right\|^2
+ \left\|\frac{Tx}{\|Tx\|} - \frac{Ty}{\|Ty\|}\right\|^2
\\ & \leq (\|T\|^2 + \mu[T]^2)\left(\frac{1}{\|Tx\|^2}
+ \frac{1}{\|Ty\|^2}\right) + (\|T\|^2 - \mu[T]^2)\frac{2c}{\|Tx\|\,\|Ty\|}
\\ & \leq (\|T\|^2 + \mu[T]^2)\left(\frac{1}{\|Tx\|^2}
+ \frac{1}{\|Ty\|^2}\right) + c(\|T\|^2 - \mu[T]^2)\left(\frac{1}{\|Tx\|^2} + \frac{1}{\|Ty\|^2}\right)
\\ & = \Big((1 + c)\|T\|^2 + (1 - c)\mu[T]^2\Big)\left(\frac{1}{\|Tx\|^2} + \frac{1}{\|Ty\|^2}\right).
\end{align*}
Hence,
\begin{align}\label{id.3.13}
\frac{1}{\|Tx\|^2} + \frac{1}{\|Ty\|^2} \geq \frac{4}{(1 + c)\|T\|^2 + (1 - c)\mu[T]^2}.
\end{align}
It follows from (\ref{id.3.12}) and (\ref{id.3.13}) that
\begin{align*}
\left\langle \frac{Tx}{\|Tx\|}, \frac{Ty}{\|Ty\|}\right\rangle - c &
= 1 - c - \frac{1}{2}\left\|\frac{Tx}{\|Tx\|} - \frac{Ty}{\|Ty\|}\right\|^2
\\ & = 1 - c + \frac{\mu c[T]^2}{\|Tx\|\,\|Ty\|}
- \frac{1}{2} \mu[T]^2\left(\frac{1}{\|Tx\|^2} + \frac{1}{\|Ty\|^2}\right)
\\ & \leq 1 - c + \frac{1}{2} \mu c[T]^2\left(\frac{1}{\|Tx\|^2} + \frac{1}{\|Ty\|^2}\right)
- \frac{1}{2} \mu[T]^2\left(\frac{1}{\|Tx\|^2} + \frac{1}{\|Ty\|^2}\right)
\\ & = 1 - c - \frac{1}{2}(1 - c) \mu[T]^2\left(\frac{1}{\|Tx\|^2} + \frac{1}{\|Ty\|^2}\right)
\\ & \leq 1 - c - \frac{2(1 - c) \mu[T]^2}{(1 + c)\|T\|^2 + (1 - c)\mu[T]^2}
\\ & =  \frac{(1 - c^2)\|T\|^2 - (1 - c^2)\mu[T]^2}{(1 + c)\|T\|^2 + (1 - c)\mu[T]^2}
\\&\hspace{8cm}(\mbox{since $1\leq \mu$})
\\ & \leq \frac{(1 - c^2)(\|T\|^2 - [T]^2)}{(1 + c)\|T\|^2 + (1 - c)[T]^2}.
\end{align*}
Hence
$$\sup\left\{\left|\frac{\langle Tx, Ty\rangle}{\|Tx\|\,|Ty\|} - c\right|; \,\,x\,
\angle_c\, y,\, \|x\| = \|y\| =1, \,x, y\in\mathscr{H}\right\}
\leq \frac{(1 - c^2)(\|T\|^2 - [T]^2)}{(1 + c)\|T\|^2 + (1 - c)[T]^2}.$$
Utilizing Lemma \ref{le.1}(ii), we get
\begin{align}\label{id.3.120}
\widehat{\varepsilon}\,(T, c) \leq \frac{(1 - c^2)(\|T\|^2 - [T]^2)}{(1 + c)\|T\|^2 + (1 - c)[T]^2}.
\end{align}
Thus, by (\ref{id.3.110}) and (\ref{id.3.120}), we conclude that
$\widehat{\varepsilon}\,(T, c) = \frac{(1 - c^2)(\|T\|^2 - [T]^2)}{(1 + c)\|T\|^2 + (1 - c)[T]^2}$.
\end{proof}
As an immediate consequence of Theorem \ref{th.4},
we get a characterization of the $(\varepsilon, c)$-angle preserving mappings.
\begin{corollary}\label{cr.4}
Suppose that $T\in \mathbb{B}(\mathscr{H}, \mathscr{K})\setminus\{0\}$ and $c\in (-1,1)$.
Then there exists an $\varepsilon \in [0, 1 + |c|)$ such that $T$ is an
$(\varepsilon, c)$-angle preserving mapping if and only if $T$ is bounded below.
\end{corollary}
\begin{corollary}\label{cr.4.5}
Let $c\in (-1,1)$ and $\varepsilon\in[0, 1 + |c|)$.
Suppose that $T\in \mathbb{B}(\mathscr{H}, \mathscr{K})\setminus\{0\}$
be an $(\varepsilon, c)$-angle preserving mapping.
Then $T$ is injective and the following statements hold:
\begin{enumerate}
\item[(i)] $\sqrt{\frac{(1 + |c|)(1 - |c| - \varepsilon)}{(1 - |c|)
(1 + |c| + \varepsilon)}}\|T\| \leq [T]$.\\

\item[(ii)] $\sqrt{\frac{(1 + |c|)(1 - |c| - \varepsilon)}{(1 - |c|)
(1 + |c| + \varepsilon)}}\|Tx\|\,\|y\| \leq \|Ty\|\,\|x\| \qquad (x, y\in \mathscr{H})$.\\

\item[(iii)] $\sqrt{\frac{(1 + |c|)(1 - |c| - \varepsilon)}{(1 - |c|)
(1 + |c| + \varepsilon)}}\|T\|\,\|x\| \leq \|Tx\| \leq \|T\|\,\|x\| \qquad (x\in \mathscr{H})$.
\end{enumerate}
\end{corollary}
\begin{proof}
Since $T$ is $(\varepsilon, c)$-angle preserving mapping we have $\widehat{\varepsilon}\,(T, c) < 1 + |c|$.
Theorem \ref{th.3}, ensures that $T$ is injective. From Theorem \ref{th.4}, we reach
$$\widehat{\varepsilon}\,(T, c) =
\frac{(1 - |c|^2)(\|T\|^2 - [T]^2)}{(1 + |c|)\|T\|^2 + (1 - |c|)[T]^2} \leq \varepsilon$$
or equivalently,
$$\sqrt{\frac{(1 + |c|)(1 - |c| - \varepsilon)}{(1 - |c|)(1 + |c| + \varepsilon)}}\|T\| \leq [T].$$
For $x, y\in \mathscr{H}$, from the above inequality, we obtain
\begin{align*}
\sqrt{\frac{(1 + |c|)(1 - |c| - \varepsilon)}{(1 - |c|)(1 + |c| + \varepsilon)}}\|Tx\|\,\|y\|
&\leq \sqrt{\frac{(1 + |c|)(1 - |c| - \varepsilon)}{(1 - |c|)(1 + |c| + \varepsilon)}}\|T\|\,\|x\|\,\|y\|
\\& \leq [T]\,\|y\|\,\|x\| \leq \|Ty\|\,\|x\|.
\end{align*}
and
\begin{align*}
\sqrt{\frac{(1 + |c|)(1 - |c| - \varepsilon)}{(1 - |c|)(1 + |c| + \varepsilon)}}\|T\|\,\|x\|
\leq [T]\,\|x\|\leq \|Tx\| \leq \|T\|\,\|x\|.
\end{align*}
\end{proof}
\begin{corollary}\label{cr.5}
Let $c\in (-1,1)$. For $T, S\in \mathbb{B}(\mathscr{H})\setminus\{0\}$ the following statements hold:
\begin{enumerate}
\item[(i)] If $T, S$ are left invertible, then $\widehat{\varepsilon}\,(ST, c) < 1 + |c|$.\\

\item[(ii)] If $S$ is a scalar multiple of an isometry, then
$\widehat{\varepsilon}\,(ST, c) = \widehat{\varepsilon}\,(T, c)$.\\

\item[(iii)] If $T^{-1}\in \mathbb{B}(\mathscr{H})\setminus\{0\}$,
then $\widehat{\varepsilon}\,(T^{-1}, c) = \widehat{\varepsilon}\,(T, c)$.
\end{enumerate}
\end{corollary}
\begin{proof}
(i) Since $T$ and $S$ are left invertible, we get $[TS]\geq [T]\,[S]>0$.
Hence, by Theorem \ref{th.4}, $\widehat{\varepsilon}\,(TS, c)< 1 + |c|$.\\
(ii) This is because of $\|S\| = [S]$, $\|ST\| = \|S\|\,\|T\|$ and $[ST] = [S]\,[T]$.\\
(iii) This is due to $\|T^{-1}\| = \frac{1}{[T]}$ and $[T^{-1}] = \frac{1}{\|T\|}$.
\end{proof}
Here is another property of the function $\widehat{\varepsilon}\,(., c)$.
\begin{corollary}\label{cr.6}
Let $c\in (-1,1)$. The function $T\mapsto \widehat{\varepsilon}\,(T, c)$ is norm continuous
at each $T\in\mathbb{B}(\mathscr{H}, \mathscr{K})$ with $[T] > 0$.
\end{corollary}
\begin{proof}
Suppose that $T_n \in \mathbb{B}(\mathscr{H}, \mathscr{K})$
such that $\lim_{n\rightarrow\infty} \|T_n - T\| = 0$.
Since $T \neq 0$, we may assume that $T_n \neq 0$ for all $n \in \mathbb{N}$. Then
$$\lim_{n\rightarrow\infty} \|T_n\| = \|T\|, \quad \lim_{n\rightarrow\infty} [T_n]
= [T] \quad \mbox{and} \quad (1 + c)\|T_n\|^2 + (1 - c)[T_n]^2 \neq 0.$$
Thus by Theorem \ref{th.4}, we get
\begin{align*}
\lim_{n\rightarrow\infty} \widehat{\varepsilon}\,(T_n, c) &= \lim_{n\rightarrow\infty}
\frac{(1 - |c|^2)(\|T_n\|^2 - [T_n]^2)}{(1 + |c|)\|T_n\|^2 + (1 - |c|)[T_n]^2}
= \frac{(1 - |c|^2)(\|T\|^2 - [T]^2)}{(1 + |c|)\|T\|^2 + (1 - |c|)[T]^2}\\
& = \widehat{\varepsilon}\,(T, c).
\end{align*}
\end{proof}
\begin{remark}
The function $\widehat{\varepsilon}\,(., c)$ is not continuous at $0$ even in the case $c = 0$.
Take any mapping $T$, which is not orthogonality preserving	. Thus $\widehat{\varepsilon}\,(T, c) \neq 0$.
Let $T_n = \frac{1}{n}T$. Then $\lim_{n\rightarrow\infty} \|T_n\| = 0$,
but for every $n$, $\widehat{\varepsilon}\,(T_n, c) = \widehat{\varepsilon}\,(T, c) \neq 0$
(see \cite[Remark 2.7]{Z.C.H.K}).
\end{remark}
Now we want to prove that every injective operator
approximately preserves orthogonality. This result will be helpful.
\begin{theorem}\label{theorem-eps-T}
Suppose that $T\in \mathbb{B}(\mathscr{H}, \mathscr{K})$ and $0<[T]\leq\|T\|$. Then $T$ satisfies
\begin{center}
$x\perp y \Longrightarrow Tx\perp^{\!\!\varepsilon_T} Ty \qquad (x, y\in \mathscr{H})$,
\end{center}
with $\varepsilon_T=1-\frac{[T]^2}{\|T\|^2}$.
\end{theorem}
\begin{proof}
Fix arbitrarily two nonzero vectors $x,y\in\mathscr{H}$
such that $x\perp y$. Since $0<[T]$, $T$ is injective.
So, it follows from Corollary \ref{cr.3.5} that
there exist unit vectors $a,b\in {\rm span}\{x,y\}$ such that
\begin{equation}\label{eq-a-b-Ta-Tb}
a\perp b, \quad Ta \perp Tb, \quad \|Ta\| \leq \|Tx\| \leq \|Tb\|
\quad \mbox{and} \quad \|Ta\| \leq \|Ty\| \leq \|Tb\|.
\end{equation}
Moreover, there exist $\alpha,\beta,\gamma,\delta\in \mathbb{R}$ such
that $x=\alpha a+\beta b$, $y=\gamma a+\delta b$. Since $x\bot y$, we have
\begin{equation}\label{eq-alpha-beta-gamma-delta}
\alpha\gamma=-\beta\delta .
\end{equation}
Furthermore,  $Tx=\alpha Ta+\beta Tb$,
$Ty=\gamma Ta+\delta Tb$. If $\alpha\beta\gamma\delta=0$, then it is easy to see that
$\langle Tx,Ty\rangle=0$. Thus, in particular, we get
$Tx\bot^{\varepsilon_T} Ty$. So, now suppose that $\alpha\beta\gamma\delta\neq 0$. Let us denote
$\theta:=\frac{\alpha}{\beta}=-\frac{\delta}{\gamma}$.
It follows from (\ref{eq-a-b-Ta-Tb}) and (\ref{eq-alpha-beta-gamma-delta}) that
\begin{align*}
\frac{|\langle Tx,Ty\rangle|}{\|Tx\|\,\|Ty\|}&=\frac{\big|\
\alpha\gamma\|Ta\|^2+\beta\delta\|Tb\|^2\
\big|}{\sqrt{|\alpha|^2\|Ta\|^2+|\beta|^2\|Tb\|^2}\,\sqrt{|\gamma|^2\|Ta\|^2+|\delta|^2\|Tb\|^2}}\\
&=\frac{\left(\|Tb\|^2-\|Ta\|^2\right)|\alpha\gamma|}{\sqrt{|\alpha|^2\|Ta\|^2+|\beta|^2\|Tb\|^2}
\, \sqrt{|\gamma|^2\|Ta\|^2+|\delta|^2\|Tb\|^2}}\\
&=\!\frac{1-\frac{\|Ta\|^2}{\|Tb\|^2}}{\frac{1}{|\alpha|}\sqrt{|\alpha|^2\frac{\|Ta\|^2}{\|Tb\|^2}\!+\!|\beta|^2}
\,\frac{1}{|\gamma|}\sqrt{|\gamma|^2\frac{\|Ta\|^2}{\|Tb\|^2}\!+\!|\delta|^2}}\\
&=\!\frac{1-\frac{\|Ta\|^2}{\|Tb\|^2}}{\sqrt{\frac{\|Ta\|^2}{\|Tb\|^2}\!
+\!\frac{1}{|\theta|^2}}\,\sqrt{\frac{\|Ta\|^2}{\|Tb\|^2}\!+\!|\theta|^2}}\\
&\leq\frac{1-\frac{\|Ta\|^2}{\|Tb\|^2}}{\sqrt{\frac{\|Ta\|^4}{\|Tb\|^4}+1}}\leq
1-\frac{\|Ta\|^2}{\|Tb\|^2}\leq 1-\frac{[T]^2}{\|T\|^2}=\varepsilon_T,
\end{align*}
whence $|\langle Tx,Ty\rangle|\leq\varepsilon_T\!\|Tx\|\,\|Ty\|$.
Thus $Tx\bot^{\!\!\varepsilon_T} Ty$.
\end{proof}
The following result can be considered as an extension of
Theorem \ref{theorem-eps-T} (in some sense).
More precisely, we show that every injective operator
approximately preserves inner product.
\begin{theorem}\label{theorem-eps-T-gamma}
Assume that $\dim\mathscr{H}\! <\!\infty$. Suppose
that $T\!\in\! \mathbb{B}(\mathscr{H}, \mathscr{K})$ and $0\!<\![T]$. Then there
exists $\gamma$ such that $T$ satisfies
\begin{equation}\label{ineq-TxTy-gamma-xy-eps-T}
|\langle Tx,Ty\rangle-\gamma\langle x,y\rangle|
\leq \left(1-\frac{[T]^2}{\|T\|^2}\right)\,\|T\|^2\,\|x\|\,\|y\| \qquad (x, y\in\mathscr{H}).
\end{equation}
Moreover, $[T]^2\leq |\gamma|\leq 2\|T\|^2-[T]^2$.
\end{theorem}
\begin{proof}
Combining Theorem \ref{theorem-eps-T} and \cite[Theorem 5.5]{Wo.4}, we
immediately get (\ref{ineq-TxTy-gamma-xy-eps-T}).
Fix $u\!\in\! \mathscr{H}$ such that $\|u\|\!=\!1$.
Putting $u$ in place of $x$ and $y$ in the above inequality in (\ref{ineq-TxTy-gamma-xy-eps-T}) we get
$|\, \|Tu\|^2\!-\!\gamma |\!\leq\! \left(1-\frac{[T]^2}{\|T\|^2}\right)\,\|T\|^2$.
Since $u$ was chosen as arbitrary unit vector, passing to the supremum and infimum over $\|u\|\!=\!1$, we
get $[T]^2\!\leq\! |\gamma|\!\leq\! 2\|T\|^2\!-\![T]^2$.
\end{proof}
To end this paper we show that
in the finite-dimensional case Corollary \ref{cr.4.5} can be strengthen  as follows.
Namely, as an immediate consequence of Corollary \ref{cr.4.5} and Theorem \ref{theorem-eps-T-gamma}, we
obtain the final result.
\begin{corollary}\label{theorem-eps-T-gamma-c}
Let $c\in (-1,1)$ and $\varepsilon\in[0, 1 + |c|)$.
Suppose that $T\in \mathbb{B}(\mathscr{H}, \mathscr{K})\setminus\{0\}$
be an $(\varepsilon, c)$-angle preserving mapping.
Assume that $\dim\mathscr{H}\! <\!\infty$.
Then there
exists $\gamma$ such that $T$ satisfies
\begin{center}
$|\langle Tx,Ty\rangle-\gamma\langle x,y\rangle|\leq \left(1-\frac{(1 + |c|)(1 - |c| - \varepsilon)}{(1 - |c|)
(1 + |c| + \varepsilon)}\right)\,\|T\|^2\,\|x\|\,\|y\| \qquad (x, y\in\mathscr{H})$.
\end{center}
Moreover, $[T]^2\leq |\gamma|\leq 2\|T\|^2-[T]^2$.
\end{corollary}
\begin{proof}
It follows from Corollary \ref{cr.4.5} that $T$ is injective. Since $\dim\mathscr{H}\! <\!\infty$, we have
$[T]>0$. From Theorem \ref{theorem-eps-T-gamma} we
have the desired assertion.
\end{proof}
 
\textbf{Acknowledgement.}
The first author is partially supported by a grant from Ferdowsi University of Mashhad (No. 2/47884).

\bibliographystyle{amsplain}

\end{document}